\documentclass[11pt,reqno]{amsart}
\usepackage{silence}
\usepackage[T1]{fontenc}
\usepackage{microtype}
\usepackage{amsmath,amssymb,amsthm,mathtools}
\usepackage{xcolor}
\usepackage[hypertexnames=false,colorlinks=true,linkcolor=blue!55!black,citecolor=blue!55!black,urlcolor=blue!55!black]{hyperref}
\usepackage[nameinlink,noabbrev]{cleveref}
\usepackage{a4wide}
\usepackage{palatino}
\usepackage{autonum}

\numberwithin{equation}{section}

\newtheorem{theorem}{Theorem}[section]
\newtheorem{proposition}[theorem]{Proposition}
\newtheorem{lemma}[theorem]{Lemma}
\newtheorem{corollary}[theorem]{Corollary}

\newcommand{\Ric}{\operatorname{Ric}}
\newcommand{\Rm}{\operatorname{Riem}}
\newcommand{\tr}{\operatorname{tr}}
\newcommand{\diver}{\operatorname{div}}
\newcommand{\Id}{\operatorname{Id}}

\newcommand{\dd}{\,d}

\newcommand{\overlinegrad}{\overline\nabla}
\DeclareRobustCommand{\rchi}{{\mathpalette\irchi\relax}}
\newcommand{\irchi}[2]{\raisebox{+1.7pt}{$#1\chi$}}

\title{Short--time existence for the Schouten flow}

\author{Giovanni Catino}
\address{Dipartimento di Matematica, Politecnico di Milano,
Piazza Leonardo da Vinci 32, 20133 Milano, Italy}
\email{giovanni.catino@polimi.it}

\author{Carlo Mantegazza}
\address{Dipartimento di Matematica e Applicazioni ``Renato Caccioppoli'', Universit\`a di Napoli Federico II, Italy}
\email[C. Mantegazza]{carlo.mantegazza@unina.it}



\date{\today}

\begin{document}

\begin{abstract}
Let $(M^n,g_0)$ be a closed smooth Riemannian manifold, with $n\geqslant 3$. We prove short--time existence and uniqueness for the ``critical'' {\em Ricci--Bourguignon flow}
\begin{equation}
\partial_t g=-2\Ric_g+\frac{R_g}{n-1}g\,.
\end{equation}
Up to a constant rescaling of time, this is the {\em Schouten flow}, since the right--hand side of the equation is $-2(n-2)$ times the Schouten tensor. At the critical parameter, however, the principal symbol of the linearization of the operator that we obtain after the usual ``DeTurck modification'' still has a zero eigenvalue. We resolve this degeneracy by adjoining an independent scalar unknown $r$, intended to represent the scalar curvature, and considering an ``extended'' system consisting of a strictly parabolic equation for the metric coupled to a transport--reaction equation for $r$ with a curvature--square forcing. The extended system belongs to the class of nonlinear composite parabolic--hyperbolic systems studied by Vol'pert and Khudyaev~\cite{VolpertKhudyaev1972}. Then, a standard compact--manifold adaptation of their theorem yields a unique smooth solution of the extended system. Finally, the difference $R_g-r$ satisfies a homogeneous linear parabolic equation, so the scalar--curvature constraint ``propagates'' and the original critical Ricci--Bourguignon flow is recovered. This settles the short--time existence and uniqueness problem at the critical ``Schouten'' value $\rho=1/(2(n-1))$, which was left open by the previous Ricci--Bourguignon theory in~\cite{CatinoEtAl2017}.
\end{abstract}

\maketitle

\section*{AI usage statement}
The strategy leading to the main argument was suggested during a series of interactions with ChatGPT~5.6~Pro. The authors subsequently checked, developed and independently verified all mathematical arguments and take full responsibility for the results and for the final manuscript.

\section{Introduction}

For any real parameter $\rho$, we consider the following {\em Ricci--Bourguignon flow}
\begin{equation}\label{eq:RB}
\partial_tg=-2\bigl(\Ric_g-\rho R_g g\bigr)\,,
\end{equation}
whose study was proposed by J.--P. Bourguignon~\cite[Question~3.24]{Bourguignon1981}. For $\rho<\frac{1}{2(n-1)}$, short--time existence and uniqueness on closed manifolds were proved in~\cite{CatinoEtAl2017} by means of the {\em DeTurck trick}~\cite{DeTurck1983}, as for the Ricci flow. At the critical value
\begin{equation}
\rho_c=\frac{1}{2(n-1)}\,,
\end{equation}
the principal symbol has one further zero eigenvalue in addition to the eigenvalues coming from the diffeomorphism invariance. Consequently, after the usual DeTurck modification is introduced, a scalar zero eigenvalue still remains and the resulting system is not strictly parabolic. This borderline case was therefore left open in~\cite{CatinoEtAl2017}. The purpose of this paper is to settle it.

The critical equation is given by
\begin{equation}\label{eq:critical}
\partial_tg=-2\Ric_g+\alpha R_g g\qquad\text{ with }\qquad \alpha=\frac{1}{n-1}
\end{equation}
and if
\begin{equation}
A_g=\frac{1}{n-2}\biggl(\Ric_g-\frac{R_g}{2(n-1)}g\biggr)
\end{equation}
denotes the Schouten tensor (see, for instance,~\cite{Besse1987}), then the right--hand side of equation~\eqref{eq:critical} equals $-2(n-2)A_g$. Therefore, after the linear time change $\tau=(n-2)t$, equation~\eqref{eq:critical} becomes the {\em Schouten flow}
\begin{equation}
 \partial_\tau g=-2A_g
\end{equation}
(notice that in dimension three no time rescaling is needed).

Our main result is the following.

\begin{theorem}\label{thm:main}
Let $M^n$ be a closed smooth manifold, $n\geqslant 3$, and let $g_0$ be a smooth Riemannian metric on $M$. There exist $T>0$ and a unique
\begin{equation}
g\in C^\infty(M\times[0,T))
\end{equation}
satisfying
\begin{equation}\label{eq:main-ivp}
\partial_tg=-2\Ric_g+\frac{R_g}{n-1}g\qquad\text{ and }\qquad g(0)=g_0\,.
\end{equation}
\end{theorem}

The central observation behind the proof is the scalar--curvature evolution. Along the general Ricci--Bourguignon flow, equation~\eqref{eq:RB}, one has
\begin{equation}
\partial_tR_g=\bigl(1-2(n-1)\rho\bigr)\Delta_gR_g+2|\Ric_g|_g^2-2\rho R_g^2
\end{equation}
(see~\cite{CatinoEtAl2017}, or directly from the scalar curvature variation formula). Hence, at the critical value, one has
\begin{equation}\label{eq:scalar-critical-formal}
\partial_tR_g=2|\Ric_g|_g^2-\alpha R_g^2\,.
\end{equation}
The idea is then to introduce an independent scalar unknown $r$, intended to represent $R_g$, and, after fixing the diffeomorphism invariance by the DeTurck trick, to couple the metric equation with the corresponding DeTurck--modified scalar--curvature evolution. This produces a coupled parabolic--transport system. The resulting local existence problem falls within the class of nonlinear composite systems studied by Vol'pert and Khudyaev, see~\cite[Theorem~1, pp.~520--521]{VolpertKhudyaev1972}. As their theorem is formulated for vector--valued functions on $\mathbb R^n$, we record below only the compact--manifold special case needed here. A closely related device was used by R{\aa}de~\cite{Rade1992} for the Yang--Mills heat flow: the curvature was treated as an independent unknown in a coupled system and was subsequently recovered by propagation of the corresponding constraint. The specific parabolic--transport structure arising here appears, however, to be new.

The proof is organized as follows. In \Cref{sec:extended} we introduce the extended DeTurck system and prove its smooth local existence and uniqueness by means of the theorem of Vol'pert and Khudyaev. In \Cref{sec:constraint} we prove that the constraint $r=R_g$ ``propagates'', recover the original flow and conclude the proof of \Cref{thm:main}.

\medskip

\noindent{\sc Acknowledgments.} The authors thank Francesca Oronzio for useful discussions and for her careful comments on a preliminary version of the manuscript.

\section{The extended DeTurck system}\label{sec:extended}

From now on we set $\overline g=g_0$. For every Riemannian metric $g$, we define the {\em DeTurck vector field} relative to $\overline g$ by
\begin{equation}\label{eq:W}
W^k(g)=g^{ij}\bigl(\Gamma^k_{ij}(g)-\Gamma^k_{ij}(\overline g)\bigr)
\end{equation}
and following the line of the {\em DeTurck trick}, we add the term $\mathcal L_{W(g)}g$ to the flow equation~\eqref{eq:main-ivp}, which then becomes
\begin{equation}
\partial_tg\ =-2\Ric_g+\alpha R_g g+\mathcal L_{W(g)}g\,.
\end{equation}
If $\varphi_s$ denotes the local flow generated by $W(g)$, then $R_{\varphi_s^*g}=\varphi_s^*R_g$ and since
\begin{equation}
\left.\frac{d}{ds}\right|_{s=0}\varphi_s^*g=\mathcal L_{W(g)}g\,,
\end{equation}
differentiating the previous identity at $s=0$, we get
\begin{equation}
DR_g[\mathcal L_{W(g)}g]=\mathcal L_{W(g)}R_g=W^i(g)\partial_iR_g=W(g)(R_g)\,.
\end{equation}
Since the scalar--curvature variation formula is linear in the metric variation, the additional 
term $\mathcal L_{W(g)}g$ in the metric equation contributes exactly $W(g)(R_g)$ to the scalar--curvature evolution. Hence, equation~\eqref{eq:scalar-critical-formal} becomes
\begin{equation}
\partial_tR_g=2|\Ric_g|_g^2-\alpha R_g^2+W(g)(R_g)\,.
\end{equation}
Then, we introduce an independent scalar unknown function $r$, intended to represent $R_g$, and we analyze the following extended system:
\begin{equation}\label{eq:extended}
\begin{cases}
&\partial_tg\ =-2\Ric_g+\alpha r g+\mathcal L_{W(g)}g\\
&\partial_tr\ =2|\Ric_g|_g^2-\alpha r^2+W(g)(r)\\
&g(0)\!=g_0\\
&r(0)\!=r_0
\end{cases}
\end{equation}
Here $r_0\in C^\infty(M)$ is, for the moment, arbitrary. The choice $r_0=R_{g_0}$ will be imposed only after the extended system has been solved. The first equation of system~\eqref{eq:extended} is a strictly parabolic quasilinear equation for the metric once $r$ is regarded as an independent variable.

\begin{lemma}[Principal part of the metric equation]\label{lem:deturck}
With respect to the background connection $\overlinegrad$, there holds
\begin{equation}\label{eq:deturck-principal}
-2\Ric_g+\mathcal L_{W(g)}g =g^{ij}\overline\nabla_i\overline\nabla_j g+Q(g,g^{-1},\overlinegrad g,\Rm_{\overline g})\,,
\end{equation}
where $Q$ depends smoothly on $g$ and $g^{-1}$ and contains at most one background derivative of $g$.
\end{lemma}
\begin{proof}
The difference of connections
\begin{equation}
T^k_{ij}=\Gamma^k_{ij}(g)-\Gamma^k_{ij}(\overline g)
\end{equation}
is a tensor and satisfies
\begin{equation}
T^k_{ij}=\frac12 g^{k\ell}\bigl(\overline\nabla_i g_{j\ell}+\overline\nabla_j g_{i\ell}-\overline\nabla_\ell g_{ij}\bigr)\,.
\end{equation}
Thus, $T$ is linear in $\overline\nabla g$, with coefficients depending
smoothly on $g^{-1}$. The difference formula for the Ricci tensors is
\begin{equation}
(\Ric_g)_{ij}= (\Ric_{\overline g})_{ij} +\overline\nabla_kT^k_{ij}
-\overline\nabla_jT^k_{ik} +T^k_{k\ell}T^\ell_{ij}-T^k_{j\ell}T^\ell_{ik}\,.
\end{equation}
It follows from the preceding expression for $T$ that the only second--order
derivatives of $g$ in $\Ric_g$ occur in the terms $\overline\nabla T$. The non--elliptic second--order terms are precisely canceled by the corresponding second--order terms in $\mathcal L_{W(g)}g$, and the remaining second--order term is $g^{ij}\overline\nabla_i\overline\nabla_j g$. This is the usual Ricci--DeTurck identity; see, for example,~\cite{DeTurck1983,Topping2006}.
\end{proof}

Consequently, writing $u=g-\overline g$, system~\eqref{eq:extended} takes the
analytic form
\begin{equation}\label{eq:analytic-system}
\begin{cases}
&\partial_tu =g^{ij}\overline\nabla_i\overline\nabla_ju+Q(g,g^{-1},\overlinegrad g,\Rm_{\overline g})+\alpha r g\\
&\partial_tr=2|\Ric_g|_g^2-\alpha r^2+W^i(g)\overline\nabla_i r
\end{cases}
\end{equation}
(observing that $W(g)(r)=W^i(g)\partial_i r=W^i(g)\overline\nabla_i r$). Thus, the metric equation is second--order parabolic, whereas the scalar equation is first--order in $r$ and depends nonlinearly on the second derivatives of the metric through the curvature term.

We now want to apply the following special case of the local existence theorem of Vol'pert
and Khudyaev~\cite[Theorem~1]{VolpertKhudyaev1972}, adapted to closed manifolds.

\begin{lemma}[Vol'pert--Khudyaev on a closed manifold]
\label{lem:volpert-khudyaev}
Let $(M,\overline g)$ be a closed smooth Riemannian manifold and let
$E\to M$ be a smooth finite--rank vector bundle endowed with a connection
$\overline\nabla$. Consider
\begin{equation}\label{eq:VK-compact-system}
\begin{cases}
\partial_tU=a^{ij}(x,U)\overline\nabla_i\overline\nabla_jU+H(x,U,\overline\nabla U,V)\\
\partial_tV=b^i(x,U,\overline\nabla U)\overline\nabla_iV+K(x,U,\overline\nabla U,\overline\nabla^2U,V)
\end{cases}
\end{equation}
where $U$ is a section of $E$, $V$ is scalar and all the coefficients are
smooth and defined for all values of their arguments. Suppose that
$a^{ij}=a^{ji}$ and, for some $\lambda>0$,
\begin{equation}\label{eq:VK-uniform-ellipticity}
a^{ij}(x,U)\xi_i\xi_j\geqslant\lambda|\xi|_{\overline g}^2
\end{equation}
for every $x\in M$, $U\in E_x$ and $\xi\in T_x^*M$. Then, any smooth initial
data admit a unique smooth solution on $M\times[0,T)$ for some $T>0$.
\end{lemma}
\begin{proof}
In local coordinates and a local frame of $E$,
system~\eqref{eq:VK-compact-system} is a special case
of~\cite[Theorem~1]{VolpertKhudyaev1972}. The matrices multiplying the
time derivatives are the identity, the characteristic matrices of the
scalar first--order equation are $1\times1$ and hence automatically
symmetric, and condition~\eqref{eq:VK-uniform-ellipticity} gives the
required parabolicity. The passage from $\mathbb R^n$ to $M$ is obtained
using a finite trivializing atlas and a partition of unity, regularizing
the local components by convolution with a fixed even smooth Friedrichs
kernel. Derivatives of the cutoffs, transition functions and connection
coefficients produce only lower--order terms, so the estimates and
uniqueness argument carry over. Smoothness of the coefficients and initial
data gives the remaining boundedness, local Lipschitz and Sobolev
assumptions. For $s>n/2+3$, the approximations converge strongly in
$C([0,T];H^{s-1})\hookrightarrow C([0,T];C^2)$, which permits passage to
the nonlinear dependence of the scalar equation on
$\overline\nabla^2U$. A standard bootstrap argument using the higher--order 
estimates yields smoothness for smooth initial data.
\end{proof}

\begin{proposition}\label{prop:extended-existence}
Let $g_0$ be a smooth Riemannian metric on the closed manifold $M$ and let
$r_0\in C^\infty(M)$. Then there exist $T>0$ and a unique smooth solution
$(g,r)$ of system~\eqref{eq:extended} on $M\times[0,T)$ with initial data
$(g_0,r_0)$.
\end{proposition}
\begin{proof}
We set
\begin{equation}
E=\operatorname{Sym}^2T^*M,\qquad U=g-\overline g,\qquad V=r.
\end{equation}
As long as $g=\overline g+U$ is positive definite,
system~\eqref{eq:analytic-system} has the form
of~\eqref{eq:VK-compact-system} with
\begin{align}
a^{ij}&=g^{ij},&
H&=Q(g,g^{-1},\overlinegrad g,\Rm_{\overline g})+\alpha Vg,\\
b^i&=W^i(g),&
K&=2|\Ric_g|_g^2-\alpha V^2.
\end{align}
Since $\Ric_g$ is affine in $\overline\nabla^2U$, the function $K$ is
generally quadratic in $\overline\nabla^2U$, as allowed in
\Cref{lem:volpert-khudyaev}. Choose $\delta>0$ such that
\begin{equation}\label{eq:VK-cutoff-region-minimal}
|U|_{\overline g}<2\delta\quad\Longrightarrow\quad\overline g/2\leqslant\overline g+U\leqslant2\overline g\,.
\end{equation}
Let $\rchi:E\to[0,1]$ be a smooth function, equal to one when $|U|_{\overline g}\leqslant\delta$, and with support compactly contained in
$\{|U|_{\overline g}<2\delta\}$. The products of $\rchi$ with $g^{ij}$, $H$, $W^i$ and $K$ extend smoothly
by zero outside the positive--metric region. Setting
\begin{align}
\widehat a^{ij}&=\rchi g^{ij}+(1-\rchi)\overline g^{ij},&
\widehat H&=\rchi H,\\
\widehat b^i&=\rchi W^i,&
\widehat K&=\rchi K,
\end{align}
these coefficients are smooth for all values of their arguments and condition~\eqref{eq:VK-cutoff-region-minimal} gives
\begin{equation}
\widehat a^{ij}\xi_i\xi_j\geqslant\frac12|\xi|_{\overline g}^2\,.
\end{equation}
Therefore \Cref{lem:volpert-khudyaev}, applied to
\begin{equation}
\begin{cases}
\partial_tU=\widehat a^{ij}\overline\nabla_i\overline\nabla_jU+\widehat H\\
\partial_tV=\widehat b^i\overline\nabla_iV+\widehat K
\end{cases}
\end{equation}
with initial data $U(0)=0$ and $V(0)=r_0$, gives a unique smooth solution
for a short time. After decreasing $T$, we have
$|U|_{\overline g}<\delta/2$ on $M\times[0,T)$, hence $\rchi=1$ and this
solution satisfies system~\eqref{eq:extended}.

If $(g',r')$ is another solution with the same initial data, we set
$U'=g'-\overline g$ and $V'=r'$. The pair $(U',V')$ solves the cutoff
system above as long as $|U'|_{\overline g}<\delta$. Since
$|U|_{\overline g}<\delta/2$, the equality of the two solutions prevents
$|U'|_{\overline g}$ from reaching $\delta$. Thus $(g',r')=(g,r)$
throughout the common interval of existence. This proves uniqueness.
\end{proof}

\section{Propagation of the constraint and recovery of the original flow}\label{sec:constraint}

We now assume $r_0=R_{g_0}$. We first show that the extended solution preserves the identity $r=R_g$, so that the metric equation reduces to the critical Ricci--Bourguignon flow equation with the DeTurck modification. We then recover the original geometric flow and conclude the proof of \Cref{thm:main}.

\begin{proposition}\label{prop:constraint}
Let $(g,r)$ be the smooth solution of system~\eqref{eq:extended} with $r(0)=R_{g_0}$. Then,
\begin{equation}
 r(t)=R_{g(t)}
\end{equation}
for every $t\in[0,T)$.
\end{proposition}
\begin{proof}
For a metric variation $v=\partial_tg$, the scalar--curvature variation formula (see, e.g.,~\cite{Besse1987}) is
\begin{equation}\label{eq:variation-R}
 \partial_tR_g =-\Delta_g(\tr_gv)+\diver_g\diver_gv-g^{ik}g^{j\ell}v_{ij}(\Ric_g)_{k\ell}\,.
\end{equation}
Setting
\begin{equation}
 v_0=-2\Ric_g+\alpha rg\,,
\end{equation}
the contracted Bianchi identity gives
\begin{equation}
\tr_gv_0=-2R_g+\frac{n}{n-1}r\qquad\text{ and }\qquad\diver_gv_0=-\dd R_g+\frac{1}{n-1}\dd r\,.
\end{equation}
Therefore, recalling that $\alpha=1/(n-1)$, we have
\begin{equation}
-\Delta_g(\tr_gv_0)+\diver_g\diver_gv_0=2\Delta_gR_g-\frac{n}{n-1}\Delta_gr-\Delta_gR_g
+\frac{1}{n-1}\Delta_gr=\Delta_g(R_g-r)\,.
\end{equation}
Setting $v=v_0+\mathcal L_{W(g)}g$, we then obtain
\begin{align}
\partial_tR_g=&\,\Delta_g(R_g-r)-g^{ik}g^{j\ell}(v_0)_{ij}(\Ric_g)_{k\ell}+W(g)(R_g)\\
=&\,\Delta_g(R_g-r)+2|\Ric_g|_g^2-\alpha rR_g+W(g)(R_g)\,,\label{eq:R-extended-evolution}
\end{align}
since the variation of the scalar curvature in the direction $\mathcal L_{W(g)}g$ is $\mathcal L_{W(g)}R_g=W(g)(R_g)$ and there holds
\begin{equation}\label{eq:inner-v}
-g^{ik}g^{j\ell}(v_0)_{ij}(\Ric_g)_{k\ell}=2|\Ric_g|_g^2-\alpha rR_g\,.
\end{equation}
Subtracting the second equation in system~\eqref{eq:extended} and setting
\begin{equation}
 z=R_g-r
\end{equation}
gives the homogeneous linear PDE
\begin{equation}\label{eq:defect}
\partial_tz=\Delta_gz+W(g)(z)-\alpha rz\qquad\text{ with }\qquad z(0)=0\,.
\end{equation}
Then, by the standard parabolic maximum principle, we conclude that $z$ is identically zero on $M\times[0,T)$, which proves the proposition.
\end{proof}

It follows that the first equation in system~\eqref{eq:extended} reduces to the critical Ricci--Bourguignon flow equation with the DeTurck modification:
\begin{equation}\label{eq:gauged-final}
\partial_tg=-2\Ric_g+\alpha R_gg+\mathcal L_{W(g)}g\,.
\end{equation}

Let $g(t)$ be the solution of equation~\eqref{eq:gauged-final} on $[0,T)$ and set $X_t=-W(g(t))$. Since $M$ is closed and $X_t$ is smooth, its time--dependent flow is defined on the whole interval $[0,T)$; equivalently, solve the ordinary differential equation
\begin{equation}\label{eq:diffeomorphism}
\partial_t\Phi_t=-W(g(t))\circ\Phi_t\qquad\text{ with }\qquad \Phi_0=\Id_M\,,
\end{equation}
then, every $\Phi_t$ is a diffeomorphism. Define
\begin{equation}
\widetilde g(t)=\Phi_t^*g(t)\,,
\end{equation}
then, 
\begin{equation}
\partial_t\widetilde g=\Phi_t^*\bigl(\partial_tg+\mathcal L_{X_t}g\bigr)
=\Phi_t^*\bigl(\partial_tg-\mathcal L_{W(g)}g\bigr)
=-2\Ric_{\widetilde g}+\alpha R_{\widetilde g}\widetilde g
\end{equation}
(where in the last equality we used the identities $\Ric_{\Phi_t^*g}=\Phi_t^*\Ric_g$ and
$R_{\Phi_t^*g}=\Phi_t^*R_g$); moreover, $\widetilde g(0)=g_0$. This proves the existence part of \Cref{thm:main}.

We then prove uniqueness: let $\widehat g(t)$ be a smooth solution of equation~\eqref{eq:main-ivp} with $\widehat g(0)=g_0$. Let $\psi$ be the solution of the {\em harmonic map flow}
\begin{equation}\label{eq:harmonic-map-heat}
\partial_t\psi=\tau_{\widehat g(t),\overline g}(\psi)
\qquad\text{ with }\qquad
\psi(0)=\Id_M\,,
\end{equation}
where $\tau_{\widehat g(t),\overline g}(\psi)$ denotes the {\em tension field} of the map
$\psi:(M,\widehat g(t))\to(M,\overline g)$ (see, for instance,~\cite{EellsSampson1964,Topping2006}).
In local coordinates, equation~\eqref{eq:harmonic-map-heat} reads
\begin{equation}
(\partial_t\psi)^m=\widehat g^{ij}\bigl(\partial_i\partial_j\psi^m
-\widehat\Gamma^k_{ij}\partial_k\psi^m+\overline\Gamma^m_{\ell s}
\partial_i\psi^\ell\partial_j\psi^s\bigr)\,,
\end{equation}
which is clearly a strictly parabolic system for $\psi$, hence it has a unique smooth short--time solution, as in the standard DeTurck uniqueness argument~\cite{DeTurck1983,EellsSampson1964,Topping2006}.\\
Since $\psi(0)=\Id_M$, after possibly reducing the interval, the map $\psi(t)$ remains a diffeomorphism, thus, setting
\begin{equation}
\varphi=\psi^{-1}\qquad\text{ and }\qquad g=\varphi^*\widehat g\,,
\end{equation}
the transformation rule for the tension field $\tau$ under diffeomorphisms of the domain gives
\begin{equation}\label{eq:tension-W}
\tau_{\widehat g,\overline g}(\psi)=\tau_{g,\overline g}(\Id_M)\circ\psi=-W(g)\circ\psi\,,
\end{equation}
because
\begin{equation}
\tau_{g,\overline g}(\Id_M)^k=g^{ij}\bigl(\overline\Gamma^k_{ij}-\Gamma^k_{ij}(g)\bigr)=-W^k(g)\,.
\end{equation}
Hence, equation~\eqref{eq:harmonic-map-heat} implies
\begin{equation}\label{eq:psi-ode}
\partial_t\psi=-W(g)\circ\psi\,,
\end{equation}
then, differentiating $\psi\circ\varphi=\Id_M$ and using equation~\eqref{eq:psi-ode}, we obtain
\begin{equation}
\partial_t\varphi=\dd\varphi\bigl(W(g)\bigr)\,.
\end{equation}
Indeed,
\begin{equation}
(\partial_t\varphi)\circ\varphi^{-1}=\varphi_*W(g),\qquad\varphi^*\mathcal L_{\varphi_*W(g)}\widehat g=\mathcal L_{W(g)}g.
\end{equation}
Consequently, the time--dependent pullback formula, together with the identities 
\begin{equation}
\Ric_{\varphi^*\widehat g}=\varphi^*\Ric_{\widehat g}\qquad\text{ and }\qquad R_{\varphi^*\widehat g}=\varphi^*R_{\widehat g}\,,
\end{equation}
yields
\begin{equation}
\partial_tg=\varphi^*(\partial_t\widehat g)+\mathcal L_{W(g)}g=-2\Ric_g+\alpha R_gg+\mathcal L_{W(g)}g\,.
\end{equation}
Hence, $g$ satisfies equation~\eqref{eq:gauged-final}, and therefore $(g,R_g)$ solves system~\eqref{eq:extended}.

Now let $\widehat g_1$ and $\widehat g_2$ be two solutions of equation~\eqref{eq:main-ivp} with the same initial metric, and perform the preceding construction for each of them. On a common short interval, the corresponding extended pairs agree by \Cref{prop:extended-existence}; denote their common metric by $g$. By equation~\eqref{eq:tension-W}, the two maps satisfy the same ordinary differential equation
\begin{equation}
\partial_t\psi_i=-W(g)\circ\psi_i\qquad\text{ with }\qquad \psi_i(0)=\Id_M\,,
\end{equation}
hence, the ODE uniqueness theorem gives $\psi_1=\psi_2$, and since $\widehat g_i=\psi_i^*g$, we conclude that
$\widehat g_1=\widehat g_2$ on this short interval.

To extend the conclusion to the whole common interval of existence, suppose that the two solutions agree up to some time $t_0$. Restart the preceding harmonic map heat construction at $t_0$, taking their common metric $\widehat g_1(t_0)=\widehat g_2(t_0)$ as the new background metric and the identity as initial map. The same local argument extends equality beyond $t_0$. A standard open--closed continuation argument therefore proves uniqueness on the whole common interval and completes the proof of \Cref{thm:main}.

\begin{corollary}[Three--dimensional Schouten flow]\label{cor:three-d}
Let $(M^3,g_0)$ be a closed smooth Riemannian three--manifold. There exists a unique smooth short--time solution of
\begin{equation}
\partial_tg=-2A_g=-2\biggl(\Ric_g-\frac14R_gg\biggr)\qquad\text{ with }\qquad g(0)=g_0\,.
\end{equation}
\end{corollary}

\end{document}